\documentclass[a4paper,11pt]{article}

\usepackage{relsize}
\usepackage[latin1]{inputenc}
\usepackage[T1]{fontenc}
\usepackage[normalem]{ulem}
\usepackage[english]{babel}
\usepackage{verbatim}
\usepackage{graphicx}
\usepackage{enumerate}
\usepackage{amsmath,amssymb,amsfonts,amsthm}
\usepackage{rotating}
\usepackage{float}
\usepackage{array}

\newtheorem{theorem}{Theorem}[section]
\newtheorem*{theoremBI}{Theorem $B^I$}
\newtheorem*{theoremBn}{Theorem $B^n$}
\newtheorem*{theoremBIG}{Theorem $B_{G}^I$}
\newtheorem{lemma}[theorem]{Lemma}

\newtheorem*{cor*}{Corollary}

\theoremstyle{definition}
\newtheorem*{defn*}{Definition}

\newtheorem{rem}[theorem]{Remark}
\newtheorem{ex}[theorem]{Example}

\input xy
\xyoption{all}

\DeclareMathOperator{\id}{id}

\DeclareMathOperator*{\colim}{colim}

\DeclareMathOperator*{\holim}{holim}

\DeclareMathOperator{\hofib}{hof}

\DeclareMathOperator{\hof}{hof}

\DeclareMathOperator{\Hom}{Hom}

\newcommand{\unders}{\!<\!}

\begin{document}
\begin{center}\LARGE{Finite Homotopy Limits of Nerves of Categories}
\end{center}

\begin{center}\large{Emanuele Dotto}

\end{center}
\vspace{.3cm}

\begin{quote}
\textsc{Abstract}. Let $I$ be a small category with finite dimensional nerve, and $X\colon I\to Cat$ a diagram of small categories.
We show that, under a ``Reedy quasi-fibrancy condition'', the homotopy limit of the geometric realization  of $X$ is itself the geometric realization of a category. This categorical model for the homotopy limit is defined explicitly, as a category of natural transformations of diagrams. For the poset $\bullet\to\bullet\leftarrow\bullet$ we recover  the model for homotopy pullbacks provided by Quillen's Theorem $B$ (specifically Barwick and Kan's version of Quillen's Theorem $B_2$). For diagrams of cubical shape, this theorem gives a criterion to determine when the nerve of a cube of categories is homotopy cartesian.

We further generalize this result to equivariant diagrams of categories. For a finite group $G$,  we show that when $X\colon I\to Cat$ has a $G$-structure in the sense of \cite{Jack} and \cite{Gdiags}, the realization of the category constructed above is weakly $G$-equivalent to the homotopy limit of the realization of $X$. For $G$-diagrams of cubical shapes, this is an equivariant version of Quillen's Theorem $B$.
\end{quote}

\tableofcontents

\section*{Introduction}

Many spaces of interest to topologists are defined as the nerve of categories.
Given a diagram of small categories $X\colon I\to Cat$, understanding the homotopy limit and the homotopy colimit of the nerve diagram  $NX\colon I\stackrel{X}{\to} Cat\stackrel{N}{\to}sSet$ is important for calculations. It is particularly useful to find categories whose nerves are weakly equivalent to these spaces. In the case of homotopy colimits, this problem has been completely solved by Thomason in \cite{Thomason}, where the author proves that the nerve of the Grothendieck construction $I\int\! X$ is always weakly equivalent to the homotopy colimit of $NX$. The analogous problem for homotopy limits is considerably harder. In his famous theorem B of \cite{Quillen}, Quillen provides a categorical model for a very specific kind of homotopy limit. He proves that the homotopy fibers of the nerve of a functor are equivalent to the nerves of the over categories, provided the functor satisfies a certain ``quasi-fibrancy'' condition. Quillen's theorem has been improved by Barwick and Kan in \cite{ClarkKan}, where the authors construct a whole family of categorical models for the homotopy pullback of two functors, under assumptions similar to Quillen's. The present paper extends these results, by constructing a categorical model for the homotopy limit of $NX$ when the indexing category $I$ has finite dimensional nerve, and $X$ satisfies a similar ``quasi-fibrancy'' condition. For cubical diagrams, this theorem gives a categorical criterion to determine if the nerve of a cube of categories is homotopy cartesian.
\vspace{.4cm}

Given two diagrams of categories  $Y,X\colon I\to Cat$ let $\Hom(Y,X)$ be the category of natural transformations from $Y$ to $X$. Its objects are natural transformations $\Phi\colon Y\to X$, and a morphism $\Lambda\colon \Phi\to \Phi'$ is a pair of natural transformations $\Lambda_1\colon Y\to Y$ and $\Lambda_2\colon X\to X$ satisfying $\Phi'\Lambda_1=\Lambda_2\Phi$.
This category is introduced in \cite{Lydakis} where the author shows, among other homotopical properties of this construction, that its nerve is isomorphic to the simplicial mapping space of natural transformations $\Hom(NY,NX)$.
For the functor $Y=I/_{(-)}\colon I\to Cat$ that sends an object $i$ to its over category $I/_i$, the nerve of $\Hom\big(I/_{(-)},X\big)$ is then isomorphic to the Bousfield-Kan formula (\cite{BK}) for the homotopy limit of $NX$. General simplicial model category theory tells us that the Bousfield-Kan formula $\Hom\big(NI/_{(-)},NX\big)$ is homotopy invariant when $NX$ is pointwise fibrant. This happens only in the rare situation when the $X_i$'s are groupoids. In the terminology used in this paper, the ``homotopy limit'' of $NX$ is the specific model
\[\holim_INX:=\Hom\big(NI/_{(-)},FNX\big)\]
for a left derived functor of the limit functor, where $NX\stackrel{\simeq}{\to}FNX$ is a pointwise fibrant replacement of $NX$ (the vertices $FNX_i$ are Kan complexes).
In our main result we find a weaker condition on $X$, similar in spirit to the condition of Quillen's theorem $B$, for which the nerve of $\Hom\big(I/_{(-)},X\big)$ is still weakly equivalent to $\holim_INX$.

Let $i\!<\!I$ be the full subcategory of the under category $i/I$ of non-identity maps with source $i$. For a diagram $X\colon I\to Cat$ we define $X_{i<}: (i\unders I)\to Cat$ as the restriction of $X$ along the projection functor $i\unders I\to I$ that sends $i\to j$ to $j$. For every object $i$ of $I$, there is a functor 
\[m_i\colon X_i\longrightarrow \Hom\big((i\unders I)/_{(-)},X_{i<}\big)\]
that sends an object $x$ of $X_i$ to the natural transformation $m_i(x)\colon (i\unders I)/_{(-)}\to X_{i<}$ consisting of the constant functors $m_i(x)_\alpha\colon (i\unders I)/_{\alpha}\to X_{j}$ that send every object to $\alpha_\ast x$. For the purpose of this paper, we say that a functor is a weak equivalence if its nerve is a weak equivalence of simplicial sets.
\begin{defn*}
A diagram $X\colon I\to Cat$ is Reedy quasi-fibrant if for every object $i$ of $I$ the functor
\[m_i/_{(-)}\colon \Hom\big((i\unders I)/_{(-)},X_{i<}\big)\longrightarrow Cat\]
sends every morphism in the category of natural transformations to a weak equivalence.
\end{defn*}
Now suppose that the nerves of the under categories $N(i/I)$ have finite dimension for every object $i$ of $I$. We call a category $I$ with this property left-finite. These categories have a canonical degree function $Ob I\to\mathbb{N}$ that sends $i$ to the dimension of $N(i/I)$. The degree function induces a filtration $I_{\leq 0}\subset I_{\leq 1}\subset \dots \subset I$, where $I_{\leq n}$ is the full subcategory of $I$ of objects of degree less than or equal to $n$. This filtration is finite precisely when $NI$ is itself finite dimensional.

\begin{theoremBI}
Let $I$ be a left-finite category, and let $X\colon I\to Cat$ be a Reedy quasi-fibrant diagram of categories. There is a weak equivalence
\[\holim_{n\in \mathbb{N}^{op}}N\Hom\big((I_{\leq n})/_{(-)},X_{\leq n}\big)\stackrel{\simeq}{\longrightarrow}\holim_I NX\]
where $X_{\leq n}$ is the restriction of $X$ to $I_{\leq n}$.
In particular if the nerve of $I$ is finite dimensional the map
\[N\Hom\big(I/_{(-)},X\big)\stackrel{\simeq}{\longrightarrow}\holim_I NX\]
is a weak equivalence.
\end{theoremBI}

Let us justify the terminology ``Reedy quasi-fibrancy'' used for $X$. The functor $m_i$ factors through the categorical limit
\[m_i\colon X_i\longrightarrow\lim\limits_{i\stackrel{\neq\id}{\to}j}X_j=\Hom(\ast,X_{i<})\longrightarrow \Hom\big((i\unders I)/_{(-)},X_{i<}\big)\]
where the second map is induced by the projection $(i\unders I)/_{(-)}\to\ast$, and the first functor is the $i$-matching functor. The nerves of the $m_i$'s are thickenings of the matching maps of $NX$.
The diagram $NX$ would be Reedy fibrant if the matching maps were Kan fibrations. The condition on $X$ of Theorem $B^I$ implies, by Quillen's Lemma \cite[p.98]{Quillen} and Thomason's theorem \cite{Thomason}, that the replacement of $m_i$ by the Grothendieck construction
\[NX_i\simeq N\big(\Hom\mbox{$\int$} m_i/_{(-)}\big)\longrightarrow N\Hom\big((i\unders I)/_{(-)},X_{i<}\big)\]
is a quasi-fibration. In this sense the condition of Theorem $B^I$ is a Reedy quasi-fibrancy condition.

Theorem $B^I$ gives a criterion to determine if the nerve of a cube of categories is homotopy cartesian. Let $\mathcal{P}(n\!+\!1)$ be the poset category of subsets of the set $\{1,\dots,n\!+\!1\}$ ordered by inclusion, and let $\mathcal{P}_0(n\!+\!1)$ be the subposet of non-empty subsets. A functor $X\colon \mathcal{P}(n\!+\!1)\to Cat$ is an $(n\!+\!1)$-cube of categories. The quasi-matching functor for $X$ at the empty vertex is the functor $m_{\emptyset}\colon X_{\emptyset}\to \Hom\big(\mathcal{P}_0(-),X_{\emptyset<}\big)$.

\begin{cor*}
Let $X\colon \mathcal{P}(n\!+\!1)\to Cat$ be Reedy quasi-fibrant. For every natural transformation $\Phi\colon \mathcal{P}_0(-)\to X_{\emptyset<}$ the nerve of the category $m_{\emptyset}/_{\Phi}$ is the total homotopy fiber of the cube $NX$ over $\Phi$. In particular if all the categories $m_{\emptyset}/_{\Phi}$ are contractible, $NX$ is homotopy cartesian.
\end{cor*}

We explain how Theorem $B^I$ relates to Quillen's Theorem B when $I$ is the poset $\bullet\to\bullet\leftarrow\bullet$. A diagram indexed over this poset is a pullback diagram of categories $C\stackrel{f}{\to} D\stackrel{g}{\leftarrow} E$. There is an isomorphism of categories  
\[\Hom\Big((\bullet\to\bullet\leftarrow\bullet)/_{(-)},C\stackrel{f}{\to} D\stackrel{g}{\leftarrow} E\Big)\cong f\!\downarrow\! g\]
where $f\!\!\downarrow\!\!g$ is the model for the homotopy pullback of Barwick and Kan \cite{ClarkKan}. The objects of $f\!\!\downarrow\!\!g$ are triples $(c,d,\gamma)$ consisting of objects $c\in C$ and $e\in E$, and a zig-zag of morphisms $\gamma=(f(c)\to d \leftarrow g(e))$ in the category $D$. This is precisely the Grothendieck construction of the functor 
\[f/_{(-)}\times g/_{(-)}\colon D\longrightarrow Cat\]
In \cite{ClarkKan} the authors prove that when one of the functors $f/_{(-)},g/_{(-)}\colon D\to Cat$ sends every morphism to an equivalence, the nerve of $f\!\downarrow\!g$ is equivalent to the homotopy pullback of the nerves. In fact, they even show that the the zig-zag in $D$ can be chosen of any given length. In particular when it has length one and $E$ is the point category, this is the classical theorem B of Quillen \cite{Quillen}.
The condition from theorem Theorem $B^I$ specified to pullback diagrams requires that both functors $f/_{(-)}$ and $g/_{(-)}$ send morphisms to weak equivalences. This is clearly more restrictive than the theorem of \cite{ClarkKan}, where only one of the two functors needs to satisfy this condition. In section \ref{higherpb} we give an alternative proof of Theorem $B^I$ for the poset $I=\mathcal{P}_0(n\!+\!1)$. This new proof requires $2^{n+1}-n-2$ conditions, against the $2^{n+1}-2$ conditions of Theorem $B^I$ for the category $\mathcal{P}_0(n\!+\!1)$. In particular for $n=1$ the unique condition is precisely the condition of \cite{ClarkKan}.
\vspace{.4cm}

Theorem $B^I$ is proved by induction on the filtration on $I$ induced by the degree function, and the argument can be carried out equivariantly.
Let $G$ be a finite group acting on a category $I$. Following \cite{Jack}, a $G$-diagram in a category $\mathcal{C}$ is a functor $Z\colon I\to \mathcal{C}$ together with natural transformations 
\[g_Z\colon Z\to Z\circ g\] for every group element $g$ in $G$, satisfying the conditions $1_Z=\id$ and $g_Z\circ h_Z=(gh)_Z$. Given two $G$-diagrams of simplicial sets $Z,L\colon I\to sSet$, the simplicial set of natural transformations of underlying diagrams $\Hom(L,Z)$ inherits a $G$-action by conjugation. In particular for a suitable pointwise fibrant replacement $Z\stackrel{\simeq}{\to} FZ$, the Bousfield-Kan homotopy limit $\holim_I Z=\Hom\big(NI/_{(-)},FZ\big)$ has a $G$-action. The homotopical properties of this construction are studied extensively in \cite{Gdiags} in a suitable model categorical context.
Given two $G$-diagrams of categories $X,Y\colon I\to Cat$,
the category of natural transformations $\Hom(Y,X)$ has a similar $G$-action defined by conjugation, and the canonical isomorphism $N\Hom(Y,X)\cong \Hom(NY,NX)$ is equivariant.

Now suppose that $I$ is left-finite. Notice that the fixed point categories $I^H$ are automatically left-finite for every subgroup $H$ of $G$. For every object $i$ of $I^H$, the under category $i/I$ has an action of $H$, that restricts to the subcategory $i\unders I$.
The restriction of a $G$-diagram $X\colon I\to Cat$ to $i\unders I$ has a canonical structure of $H$-diagram, and the functor $m_i\colon X_i\rightarrow \Hom\big((i\unders I)/_{(-)},X_{i<}\big)$ is $H$-equivariant. Let
\[m^{H}_i\colon X^{H}_i\longrightarrow \Hom\big((i\unders I)/_{(-)},X_{i<}\big)^{H}\] be its restriction to the fixed point categories.

\begin{theoremBIG}
Let $I$ be a left finite category with $G$-action, and $X\colon I\to Cat$ a $G$-diagram of categories. If for every subgroup $H$ of $G$ and every object $i$ of $I^H$ the functor $m^{H}_i/_{(-)}$
sends every morphism to a weak equivalence of categories, there is a weak $G$-equivalence
\[\holim_{n\in \mathbb{N}^{op}}N\Hom\big((I_{\leq n})/_{(-)},X_{\leq n}\big)\stackrel{\simeq}{\longrightarrow}\holim_I NX\]
In particular if the nerve of $I$ is finite dimensional the map $N\Hom\big(I/_{(-)},X\big)\stackrel{\simeq}{\longrightarrow}\holim_I NX$
is a weak $G$-equivalence of simplicial $G$-sets.
\end{theoremBIG}

\section{Left-finite homotopy limits of categories}

In this section we give a proof of Theorem $B^I$ from the introduction.
A left-finite category $I$ has a canonical degree function $\deg\colon Ob I\to\mathbb{N}$, that sends an object $i$ to the dimension of $N(i/I)$. This is the length of the longest sequence of non-identity morphisms starting at $i$. This function is strictly decreasing, in the sense that if $\alpha\colon i\to j$ is a non-identity morphism $\deg(i)>\deg(j)$. In particular $I$ cannot have non-trivial endomorphisms.
The degree function induces a filtration
 \[ I_{\leq 0}\subset I_{\leq 1}\subset \dots\subset I_{\leq n}\subset I_{\leq n+1}\subset\dots\subset I\]
where $I_{\leq n}$ is the subposet of $I$ of objects of degree less than or equal to $n$.
For degree reasons there cannot be a non-identity map between objects with the same degree. Hence the complements $I_n=I_{\leq n}\backslash I_{\leq n-1}$ are discrete categories, and in particular $I_{\leq 0}=I_0$ is discrete. The proof of Theorem $B^I$ is by induction on this filtration, exploiting the fact that the categories $I_n$ are discrete.

We need a Lemma to carry over the induction step, and this requires to set up some notation.
For any set of degree $n$ objects $U\subset I_n$, let $U\!\leq\! I$ be the union of the under categories $u/I$ for $u\in U$. Explicitly, its set of objects is
\[Ob\ (U\!\leq\! I)=\{(u\in U,\alpha \colon u\to i)\}\]
The set of morphisms $(u,\alpha)\to (v,\beta)$ is empty if $u$ and $v$ are different, and it is the set of morphisms $(u,\alpha)\to (u,\beta)$ in $u/I$ otherwise.
Define $U\unders I$ to be the full subcategory of $U\!\leq\! I$ of non-identity maps.
Given a diagram of categories $X\colon I\to Cat$, we denote the corresponding restrictions by
\[X_{U\leq}\colon U\!\leq\! I\to I\stackrel{X}{\to} Cat \ \ \ \ \ \ \ \ \ \ \ X_{U<}\colon U\unders I\to I\stackrel{X}{\to} Cat\] where $U\!\leq\! I\to I$ and  $U\unders I\to I$ project onto the target. For an element $u\in U$, recall that $m_u\colon X_u\to \Hom\big((u\unders I)/_{(-)},X_{u<}\big)$ is the functor that sends an object $x$ to the natural transformation consisting of constant functors $m_u(x)_\alpha=\alpha_\ast x$. Finally, we recall from \cite{Thomason} that the Grothendieck construction of a functor $F\colon K\to Cat$ is the category $K\int F$ with objects pairs $(k\in K,x\in F(k))$, and where a morphism $(k,x)\to (l,y)$ the set is a morphism $\alpha\colon k\to l$ in $K$ together with a morphism $\delta\colon \alpha_\ast x\to y$ in the category $F(l)$.

\begin{lemma}\label{indgrot}
Let $X\colon I\to Cat$ be a diagram of categories, and suppose that $I$ is left-finite. For every subset $U\subset I_n$, there is a natural isomorphism of categories
\[\Hom\big(U\!\leq\! I)/_{(-)},X_{U\leq}\big)\cong \Big(\Hom\big((U\unders I)/_{(-)},X_{U<}\big)\mbox{$\int$}F_U\Big)\]
where $F_U\colon \Hom\big((U\unders I)/_{(-)},X_{U<}\big)\to Cat$ is the functor that sends a natural transformation $\Phi$ to the category
\[F_U(\Phi)=\prod_{u\in U}(m_u)/_{(\Phi|_{u<I})}\]
\end{lemma}

\begin{proof}
An object in the Grothendieck construction is a collection of functors $\{\Phi_{\alpha}\colon (U\unders I)/_{\alpha}\to X_{i}\}_{\alpha}$ natural in $\alpha\colon u\to i$ ranging over the objects of $U\unders i$, together with  objects $x_u\in X_u$ for every $u\in U$, and compatible natural transformations for every $\alpha\colon u\to i$
\[\gamma_{\alpha}\colon \alpha_\ast x_u\longrightarrow \Phi_\alpha\]
 Here $\alpha_{\ast}x_u\colon (U\unders I)/_{\alpha}\to X_{i}$ is the constant functor with value $\alpha_{\ast}x_u$. Given such an object $(\Phi,\underline{x},\underline{\gamma})$, define a natural transformation $\Psi\colon (U\!\leq\! I)/_{(-)}\to X_{U\leq}$ as follows.

An object of $(U\!\leq\! I)/_{\alpha}$ is a factorization $\vcenter{\hbox{\xymatrix@=5pt{u\ar[dr]\ar[rr]^{\alpha}&&i\\
&k\ar[ur]}}}$, and an object of $(U\unders I)/_{\alpha}$ is a similar factorization where the map $u\to k$ is not an identity. The functor $\Psi_\alpha\colon (U\!\leq\! I)/_{\alpha}\to X_{i}$ is defined on objects by
\[\Psi_\alpha(\vcenter{\hbox{\xymatrix@=5pt{u\ar[dr]
\ar[rr]^{\alpha}&&i\\
&k\ar[ur]}}})=\left\{\begin{array}{ll}\alpha_{\ast}x_u&\ , \mbox{if } (u\to k)=\id_u\\
\Phi_\alpha(\vcenter{\hbox{\xymatrix@=5pt{u\ar[dr]
\ar[rr]^{\alpha}&&i\\
&k\ar[ur]}}})&\ , \mbox{if } (u\to k)\neq \id_u
\end{array}\right.\]
The point here is that $\Phi_\alpha(\vcenter{\hbox{\xymatrix@=5pt{u\ar[dr]
\ar[rr]^{\alpha}&&i\\
&k\ar[ur]}}})$ is defined precisely when $u\to k$ is not an identity.
\vspace{-.4cm}
A morphism $\vcenter{\hbox{\xymatrix@=5pt{u\ar[dr]
\ar[rr]^{\alpha}&&i\\
&k\ar[ur]}}}\to \vcenter{\hbox{\xymatrix@=5pt{u\ar[dr]
\ar[rr]^{\alpha}&&i\\
&l\ar[ur]}}}$ in $(U\!\leq\! I)/_\alpha$ is a map $k\to l$ such that the two relevant triangles commute. Such a morphism is sent to
\[\Psi_\alpha\Big(\!\!\vcenter{\hbox{
\xymatrix@C=5pt@R=8pt{u\ar[ddr]\ar[dr]
\ar[rr]^{\alpha}&&i\\
&k\ar[d]\ar[ur]\\
&l\ar[uur]}}}\Big)=\left\{\begin{array}{ll}
\id_{\alpha_{\ast}x_u}&, \mbox{if } (u\to l)=\id_u
\\
\gamma_{\alpha}\colon \alpha_\ast x_u\rightarrow \Phi_\alpha(\vcenter{\hbox{\xymatrix@=5pt{u\ar[dr]
\ar[rr]^{\alpha}&&i\\
&l\ar[ur]}}})&, \mbox{if } (u\to l)\neq \id_u  ,\ (u\to k)= \id_u
\\
\Phi_\alpha\big(\vcenter{\hbox{\xymatrix@=8pt{k\ar[d]\\ l}}}\big)\colon \Phi_\alpha\big(\vcenter{\hbox{\xymatrix@=5pt{u\ar[dr]
\ar[rr]^{\alpha}&&i\\
&k\ar[ur]}}}\big)\to\Phi_\alpha\big(\vcenter{\hbox{\xymatrix@=6pt{u\ar[dr]
\ar[rr]^{\alpha}&&i\\
&l\ar[ur]}}}\big)
&, \mbox{if } (u\to l)\neq \id_u ,\ (u\to k)\neq \id_u
\end{array}\right.\]
Notice that if $u\to l$ is the identity map on $u$, by degree reasons both $u\to k$ and $k\to l$ must be identities.
This procedure defines a functor 
\[\Big(\Hom\big((U\unders I)/_{(-)},X_{U<}\big)\mbox{$\int$}F_U\Big)\longrightarrow \Hom\big((U\!\leq\! I)/_{(-)},X_{U\leq}\big)\]
on objects. Extend this on morphisms as follows. Unraveling the definitions of the Grothendieck construction and of the natural transformations category, we see that a morphism $(\Phi,\underline{x},\underline{\gamma})\to (\Phi',\underline{x}',\underline{\gamma}')$ in the left-hand category is a collection of compatible natural transformations $\lambda_{\alpha}\colon \Phi_\alpha\to \Phi_{\alpha}'$, for every non-identity map $\alpha\colon u\to i$ with $u\in U$, together with morphisms $f_u\colon x_u\to x_u'$ in $X_u$ for every $u\in U$, making the squares
\[\xymatrix{\alpha_\ast x_u\ar[r]^{\alpha_\ast f_u}\ar[d]_{\gamma_{\alpha}}&\alpha_\ast x'_u\ar[d]^{\gamma'_{\alpha}}\\
\Phi_\alpha\ar[r]_{\lambda_{\alpha}}&\Phi'_{\alpha}
}\]
commute. Such a pair $(\lambda,\underline{f})$ induces a morphism $\Psi\to \Psi'$ between the associated natural transformations in $\Hom\big((U\!\leq\! I)/_{(-)},X_{U\leq}\big)$, defined at a non-identity morphism $\alpha\colon u\to i$ by
\[\Psi_\alpha=\Phi_\alpha\stackrel{\lambda_\alpha}{\longrightarrow}\Phi'_\alpha=\Psi'_\alpha\]
and at an identity map $\id_u$ by $f_u\colon \Psi_{\id_u}=\alpha_\ast x_u\to\alpha_\ast x'_u=\Psi'_{\id_u}$.
The resulting functor is an isomorphism of categories. Its inverse sends a natural transformation $\{\Psi_\alpha\colon (U\!\leq\! I)/_{\alpha}\to X_i\}_{\alpha\colon u\to i}$ to the triple $(\Phi,\underline{x},\underline{\gamma})$ consisting of the restrictions $\Phi_\alpha\colon (U\unders I)/_{\alpha}\to (U\!\leq\! I)/_\alpha\stackrel{\Psi_\alpha}{\to}X_i$ for each $(\alpha\colon u\to i)\in U\unders I$, the objects $x_u=(\Psi_u\colon\ast=(U\!\leq\! I)/_{\id_u}\to X_u)$, and the natural transformations $\gamma_{\alpha}$ defined at an object $\vcenter{\hbox{\xymatrix@=5pt{u\ar[dr]\ar[rr]^{\alpha}&&i\\
&k\ar[ur]}}}$ of $(U\!\leq\! I)/_\alpha$ by the morphism in $X_i$
 \[\alpha_\ast x_u=\alpha_\ast\Psi_{\id_u}(\vcenter{\hbox{\xymatrix@=4pt{u\ar@{=}[dr]\ar@{=}[rr]&&u\\
&u\ar@{=}[ur]}}})
=\Psi_\alpha(\vcenter{\hbox{\xymatrix@=5pt{u\ar@{=}[dr]\ar[rr]^{\alpha}&&i\\
&u\ar[ur]_{\alpha}}}})
\stackrel{}{\longrightarrow}
\Psi_\alpha(\vcenter{\hbox{\xymatrix@=5pt{u\ar[dr]\ar[rr]^{\alpha}&&i\\
&k\ar[ur]}}})=\Phi_\alpha(\vcenter{\hbox{\xymatrix@=5pt{u\ar[dr]\ar[rr]^{\alpha}&&i\\
&k\ar[ur]}}})\]
Here the second equality holds by naturality of $\Psi$, and the arrow is $\Psi_{\alpha}$ applied to the morphism $\vcenter{\hbox{\xymatrix@C=9pt@R=7pt{u\ar[ddr]\ar@{=}[dr]
\ar[rr]_{\alpha}&&i\\
&u\ar[d]\ar[ur]\\
&k\ar[uur]}}}$
of $(U\!\leq\! I)/_{\alpha}$ induced by the factorization $\vcenter{\hbox{\xymatrix@=5pt{u\ar[dr]\ar[rr]^{\alpha}&&i\\
&k\ar[ur]}}}$. The inverse can be extended similarly to morphisms.
\end{proof}

\begin{proof}[Proof of Theorem $B^I$]
Let $NX\to FNX$ be a pointwise fibrant replacement of $NX$ in the projective model structure.
We prove just below, by induction on $n$, that for every subset $U\subset I_n$ the map
\[N\Hom\big((U\!\leq\! I)/_{(-)},X_{U\leq}\big)\cong \Hom\big(N(U\!\leq\! I)/_{(-)},NX_{U\leq}\big)\to\holim_{U\leq I}(FNX)_{U\leq}\]
is a weak equivalence. In particular by choosing $U=I_n$ the category $I_n/I$ is $I_{\leq n}$, and the  equivalence above is $N\Hom\big((I_{\leq n})/_{(-)},X_{\leq n}\big)\to \displaystyle\holim_{I_{\leq n}} NX_{\leq n}$. If $NI$ is finite dimensional $I=I_{\leq d}$ for some integer $d$, and therefore $N\Hom(I/_{(-)},X)\to \displaystyle\holim_{I} NX$ is an equivalence. When $I$ is infinite, taking the homotopy limit over the maps induced by the filtration induces an equivalence
\[\holim_{n\in \mathbb{N}^{op}}N\Hom\big((I_{\leq n})/_{(-)},X_{\leq n}\big)\stackrel{\simeq}{\longrightarrow}\holim_{n\in \mathbb{N}^{op}}\holim_{I_{\leq n}} NX_{\leq n}\]
The structure maps $\holim_{I_{\leq n}} NX_{\leq n}\to \holim_{I_{\leq n-1}} NX_{\leq n-1}$ in the right-hand tower are Kan fibrations. Indeed, they are induced by mapping the cofibrations of diagrams of simplicial sets $\iota_{n}/_{(-)}\to I_{\leq n}$, where $\iota_n\colon I_{\leq n-1}\to I_{\leq n}$ is the inclusion, into the fibrant diagram $FNX_{\leq n}$. Hence the right-hand homotopy limit is equivalent to the categorical limit. Now each $\Hom\big(N(I_{\leq n})/_{(-)},FNX_{\leq n}\big)$ is isomorphic to $\Hom\big(Nj_n/_{(-)},FNX\big)$, where $j_n\colon I_{\leq n}\to I$ is the inclusion. The right-hand limit is then
\[\lim_{n\in \mathbb{N}^{op}}\Hom\big(N(I_{\leq n})/_{(-)},FNX_{\leq n}\big)\cong \Hom\big(\colim_{n}Nj_n/_{(-)},FNX\big)\cong\holim_INX\]
The last isomorphism holds as the category $j_{n}/_{i}$ includes in $j_{n+1}/_{i}$ for every object $i$ of $I$, with union $\bigcup_{n\in\mathbb{N}}j_{n}/_{i}=I/_i$. This would finish the proof of Theorem $B^I$.

It remains to prove the inductive statement above.
The base induction step $n=0$, relies on the fact that for a subset $U\subset I_0$, the category $(U\!\leq\! I)$ is discrete, with objects the identity maps $\id_u$, for $u\in U$. Additionally the category
 $(U\!\leq\! I)/_{\id_u}=\{\id_u\}$ is the one point category. Therefore $\Hom\big((U\!\leq\! I)/_{(-)},X_{U\leq}\big)$ is the product category
\[\Hom\big((U\!\leq\! I)/_{(-)},X_{U\leq}\big)=\prod_{u\in U}X_u\]
and the homotopy limit of $NX_{U\leq}$ is the product
\[\holim NX_{U\leq}=\prod_{u\in U}FNX_u\]
Since the product of simplicial sets preserve all equivalences (not only between fibrant objects), the map $N\prod_{u\in U}X_u\to \prod_{u\in U}FNX_u$ is an equivalence.

Now suppose that $N\Hom\big((U\!\leq\! I)/_{(-)},X_{U\leq}\big)\to\holim NX_{U\leq}$ is an equivalence for every subset $U\subset I_n$, and let $V$ be a subset of $I_{n+1}$. Let $\Phi\colon (V\unders I)/_{(-)}\to X_{V<}$ be a natural transformation, and consider the commutative diagram
\[\xymatrix{
NF_{V}(\Phi)\ar[rr]\ar[d]&&\prod\limits_{v\in V}\hof_{\Phi|_{v<I}}\big( NX_v\to \holim NX_{v<}\big)\ar[d]\\
N \big(\Hom\int F_V\big)\ar[r]^-{\cong}\ar[dr]&N\Hom\big((V\!\leq\! I)/_{(-)},X_{V\leq}\big)\ar[d]\ar[r]&\holim NX_{V\leq}\ar[d]\\
&N\Hom\big((V\unders I)/_{(-)},X_{V<}\big)\ar[r]^-\simeq&\holim NX_{V<}
}\]
The bottom map is an equivalence by the inductive hypothesis, since $V\unders I=U\!\leq\! I$ for the set of objects $U:=(V\unders I)\cap I_{n}$. The isomorphism is from lemma \ref{indgrot}. By our assumption on the diagram $X$, the functor $F_V$ sends every morphism to a weak equivalence. Thus by Quillen's lemma \cite[p.98]{Quillen} and Thomason's theorem \cite{Thomason}, the left hand vertical sequence is a fiber sequence.  Let $\iota \colon V\unders I\to V\!\leq\! I$ be the inclusion. The map induced by $\iota$ on homotopy limits is the restriction
\[\holim_{V\!\leq\! I}NX_{V\leq}\longrightarrow \Hom\big(N\iota/_{(-)},FNX_{V\leq}\big)\cong \holim_{V\unders I} NX_{V<}\]
along the inclusion $\iota/_{(-)}\to (V\!\leq\! I)/_{(-)}$. Since this is a cofibration of diagrams of simplicial sets and $FNX_{V\leq}$ is fibrant, the restriction map is a Kan fibration.
Its point fiber over $\Phi$ is the product of total homotopy fibers
\[\prod_{v\in V}\hof_{\Phi|_{v<I}}\big( NX_v\to \holim NX_{v<}\big)\]
and therefore the right-hand vertical sequence in the diagram above is also a fiber sequence.
We are left with showing that the map on homotopy fibers
 \[NF_V(\Phi)=\prod_{v\in V}N\big( X_v\stackrel{m_v}{\to} \Hom\big((v\unders I)/_{(-)},X_{v<}\big)\big)/_{\Phi|_{v\unders I}}\longrightarrow \prod_{v\in V}\hof_{\Phi|_{v<I}}\big( NX_v\to \holim NX_{v<}\big)\]
is an equivalence. The components of this map factor as
\[\xymatrix{N\big(X_v\stackrel{m_v}{\to}\Hom\big((v\unders I)/_{(-)},X_{v<}\big)\big)/_{\Phi|_{v<I}}\ar[r]\ar[dr]&\hof_{\Phi|_{v<I}}\big( NX_v\to N\Hom\big((v\unders I)/_{(-)},X_{v<}\big)\big)\ar[d]\\
&\hof_{\Phi|_{v<I}}\big( NX_v\to \holim NX_{v<}\big)}\]
Our assumption on $X$ says that the functor $m_v/_{(-)}$ sends every map to a weak equivalence. Hence by Quillen's theorem B the top horizontal map in the triangle is an equivalence. The vertical map is also an equivalence, as by hypothesis of induction for the set $U:=(v\unders I)\cap I_{n}$ the map $N\Hom\big((v\unders I)/_{(-)},X_{v<}\big)\big)\to \holim NX_{v<}$ is an equivalence.
\end{proof}

\begin{cor*}
Let $X\colon\mathcal{P}(n+1)\to Cat$ be a Reedy quasi-fibrant cube. Then the total homotopy fibers of $NX$ are equivalent to the nerves of the over categories $m_{\emptyset}/_{\Phi}$, for natural transformations $\Phi\colon \mathcal{P}_0(-)\to X_{\emptyset <}$. In particular if the categories $m_{\emptyset}/_{\Phi}$ are contractible $NX$ is homotopy cartesian.
\end{cor*}
\begin{proof}
Let us recall that the total homotopy fiber of $NX$ over $\Phi$ is the homotopy fiber
\[\hofib_\Phi\big(NX_{\emptyset}\longrightarrow \holim_{\mathcal{P}_0(n+1)}NX_{\emptyset <}\big)\]
Clearly the restriction $X_{\emptyset <}$ of $X$ to $\mathcal{P}_0(n+1)$ is also Reedy quasi-fibrant, and by Theorem $B^I$ the total homotopy fiber is equivalent the homotopy fiber of
\[\hofib_\Phi\big(NX_{\emptyset}\stackrel{Nm_{\emptyset}}{\longrightarrow} N\Hom\big(\mathcal{P}_0(-),X_{\emptyset <}\big)\big)\]
Since $m_{\emptyset}/_{(-)}$ also sends all maps to equivalences, this is equivalent to $Nm_{\emptyset}/_{\Phi}$ by Quillen's Theorem $B$.
\end{proof}

\section{Higher homotopy pull-backs of categories}\label{higherpb}

We weaken the hypotheses of Theorem $B^I$ when $I$ is a punctured $(n\!+\!1)$-dimensional cube. It is going to be convenient to choose a basepoint in the set indexing our cube. We replace the set $\{1,\dots,n\!+\!1\}$ from the introduction with the set $n_+=\{1,\dots,n,+\}$. Let $\mathcal{P}_0(n_+)$ be the poset of non-empty subsets of $n_+$ ordered by inclusion. The idea of the alternative proof is to express the homotopy limit of a punctured cube as an iterated homotopy pullback, and use repeatedly theorem $B$ of \cite{ClarkKan}. We do this by restricting a punctured cube $X\colon \mathcal{P}_0(n_+)\to Cat$ along a functor
\[\lambda\colon \prod_{i=1}^n\mathcal{P}_0(\{i\}_+)\to\mathcal{P}_0(n_+)\] This functor is best described as the composite
\[\xymatrix{\prod\limits_{i=1}^n\mathcal{P}_0(\{i\}_+)\ar[rr]^-{\prod\limits_i(\{i\}_+\backslash (-))}&&\big(\prod\limits_{i=1}^n\mathcal{P}(\{i\}_+)\backslash\{i\}_+\big)^{op}\ar[r]^-{U^{op}}&(\mathcal{P}_0(n_+)\backslash n_+)^{op}\ar[rr]^-{n_+\backslash (-)}&&\mathcal{P}_0(n_+)}\]
where the first and last functors are the standard isomorphisms that take complements respectively inside the $\{i\}_+$'s and $n_+$. The dual functor $U$ sends a collection of proper subsets $\underline{V}=\{V_i\subset \{i\}_+\}_{i\in \underline{n}}$ to their union in $n_+$ (having a distinct basepoint is crucial for defining this functor). 

\begin{lemma}
The functor $\lambda$ is left cofinal, that is the over categories $\lambda/S$ are contractible for every proper subset $S$ of $n_+$.
\end{lemma}

\begin{proof}
We prove equivalently that the functor $U$ is right cofinal.
Given a proper subset $S\subset n_+$, define a collection of proper subsets $\underline{S}$ in $\prod_{i=1}^n\mathcal{P}_0(\{i\}_+)\backslash \{i\}_+$, with components
\[S_i=\left\{\begin{array}{cl} \{i\}&, \mbox{if}\ i\in S\\
\emptyset&, \mbox{if}\ i\notin S\ \mbox{and}\ +\notin S\\
\{+\}&, \mbox{if}\ i\notin S\ \mbox{and}\ +\in S
\end{array}\right.\]
We claim that this is a well defined object of the under category $S/U$. Every element $i\in\underline{n}$ that is in $S$ is clearly contained in $U(\underline{S})$. In case the basepoint $+$ is in $S$, there must be an $i\in\underline{n}$ that does not belong to $S$ or $S$ wouldn't be proper. For that element $i$ the component $S_i=\{+\}$ and therefore $+$ belongs to $U(\underline{S})$. It is easy to see that $\underline{S}$ is initial in $S/U$, and therefore $S/U$ is contractible.
\end{proof}

For every integer $0\leq k\leq n$, we define $\underline{k}=\{1,\dots k\}$, with the convention that $\underline{0}$ is the empty set. Translation by a non-empty subset $K$ of $\underline{n}\backslash \underline{k}$ induces a punctured $(k+1)$-cube $X_{K\cup (-)}\colon \mathcal{P}_0(k_+)\to Cat$. Since $K$ maps uniquely to every $K\cup S$, for every subset $S$ of $\underline{k}$, there is an associated functor 
\[c_{K}\colon X_K\to \Hom\big(\mathcal{P}_0(k_+)/_{(-)},X_{K\cup (-)}\big)\] that sends an object $x$ in $X_K$ to the natural transformation $c_K(x)\colon \mathcal{P}_0(-)\to X_{K\cup (-)}$ consisting of the constant functors with value $\iota_{\ast}(x)$, where $\iota\colon K\to K\cup S$ is the inclusion. The functor $c_K$ is the functor $m_{K}$ from theorem Theorem $B^I$, for the diagram obtained by restricting $X$ to the subcategory $\mathcal{P}_0(K\cup k_+)$ of $\mathcal{P}_{0}(n_+)$.

\begin{theoremBn} Let $X\colon \mathcal{P}_0(n_+)\to Cat$ be a punctured $(n+1)$-cube of categories.
Suppose that for every integer $0\leq k\leq n-1$ and every non-empty subset $K\subset \underline{n}\backslash \underline{k}$ the functor
\[c_K/_{(-)}\colon \Hom\big(\mathcal{P}_0(k_+)/_{(-)},X_{K\cup (-)}\big)\longrightarrow Cat\]
sends every morphism to a weak equivalence. Then the nerve of $\Hom\big(\mathcal{P}_0(-),X\big)$ is equivalent to the homotopy limit of the nerve of $X$.
\end{theoremBn}

\begin{rem}
Theorem $B^n$ has one condition for every non-empty subset of $\underline{n}\backslash \underline{k}$, and every positive $k$ smaller than $n-1$. These are 
\[\sum_{k=0}^{n-1}(2^{n-k}-1)=2^{n+1}-n-2\] conditions, which is smaller than the number of conditions of Theorem $B^I$ for the category $I=\mathcal{P}_0(n_+)$, which is $2^{n+1}-2$. In particular for $n=1$, Theorem $B^n$ has a single condition requiring that the functor $c_{1}/_{(-)}$ sends morphisms to equivalences. This is the same as the condition of Theorem $B_2$in \cite{ClarkKan}.
\end{rem}

\begin{proof}[Proof of Theorem $B^n$]
As $\lambda$ is left cofinal, the induced map
\[\holim_{\mathcal{P}_0(\{n\}_+)}\dots \holim_{\mathcal{P}_0(\{1\}_+)} \lambda^\ast NX\cong\holim_{\prod_{i=1}^n\mathcal{P}_0(\{i\}_+)} \lambda^\ast NX{\longrightarrow} \holim_{\mathcal{P}_0(n_+)} NX\]
is a weak equivalence by the cofinality theorem for homotopy limits (see e.g. \cite[19.6.7]{Hirsch}). We show by inductively calculating the homotopy pullbacks that the left-hand side is equivalent to the nerve of $\Hom(\mathcal{P}_0(-),X)$. For every $1\leq k\leq n$, let $\lambda_k$ be the functor
\[\lambda_k\colon \mathcal{P}_0(k_+)\times\prod_{i=k+1}^n\mathcal{P}_0(\{i\}_+)
\longrightarrow \mathcal{P}_0(k_+)\times \mathcal{P}_0((\underline{n}\backslash\underline{k})_+)\longrightarrow \mathcal{P}_0(n_+)\]
The first map is the identity on $\mathcal{P}_0(k_+)$ crossed with the dual of the union functor, and the second map is again the dual of the union. In symbols 
\[\lambda_k(W,\{V_i\}_{i=k+1}^n)=n_+\backslash\big((k_+\backslash W)\cup (\bigcup_{i=k+1}^n(\{i\}_+\backslash V_i))\big)=\big((\underline{n}\backslash \underline{k})\cup W\big)\cap r_k(\underline{V})\]
where $r_k(\underline{V}):=n_+\backslash (\bigcup_{i=k+1}^n(\{i\}_+\backslash V_i))$. Notice that $r_k(\underline{V})$  always contains $\underline{k}$ as a proper subset.
For $k=1$ we recover our original functor $\lambda_1=\lambda$, and for $n=k$ the functor $\lambda_n$ is the identity of $\mathcal{P}_0(n_+)$.
We show by induction on $k$ that for every collection of subsets $\underline{V}=\{\emptyset \neq V_i\subset \{i\}_+\}_{i=k+1}^n$ the canonical map
\[N\Hom\big(\mathcal{P}_0(k_+)/_{(-)},X_{\lambda_k(-,\underline{V})}\big)\longrightarrow \holim_{\mathcal{P}_0(\{k\}_+)}\dots \holim_{\mathcal{P}_0(\{1\}_+)}  NX_{\lambda(-,\underline{V})}\]
is an equivalence. When $k=n$ this proves the theorem.

For $k=1$, we need to show that for every $\underline{V}=\{V_i\}_{i=2}^n$ the nerve of $\Hom\big(\mathcal{P}_0(1_+)/_{(-)},X_{\lambda(-,\underline{V})}\big)$ is equivalent to
\[\holim\big( NX_{\lambda(\{+\},\underline{V})}\stackrel{f_+}{\longrightarrow} NX_{\lambda(\{1\}_+,\underline{V})}\stackrel{f_1}{\longleftarrow} NX_{\lambda(\{1\},\underline{V})} \big)\]
This is the case by Quillen's theorem $B_2$ of \cite{ClarkKan} if $f_1/_{(-)}$ sends every morphism in $X_{\lambda(\{1\}_+,\underline{V})}$ to a weak equivalence. In the category $\mathcal{P}_0(n_+)$, the zig-zag $\lambda(\{+\},\underline{V})\to\lambda(\{1\}_+,\underline{V})\leftarrow \lambda(\{1\},\underline{V})$ is
\[(\underline{n}\backslash\underline{1})_+\cap r_1(\underline{V})\to r_1(\underline{V})\leftarrow \underline{n}\cap r_1(\underline{V})\]
If $+$ does not belong to $r_1(\underline{V})$ the right-hand map is the identity, the functor $f_1$ is the identity functor, and $f_1/_{(-)}$ sends morphisms to equivalences automatically. When $r_1(\underline{V})$ contains $+$, the right-hand map is the inclusion $r_1(\underline{V})\backslash +\to r_1(\underline{V})$. This is of the form $K\to K\cup +$ for the non-empty subset $K:=r_1(\underline{V})\backslash +\subset\underline{n}$ (it contains at least $1$). Hence $(f_1)/_{(-)}=(c_K)/_{(-)}$ sends morphisms to equivalences by assumption.

Now suppose that 
\[N\Hom\big(\mathcal{P}_0(k_+)/_{(-)},X_{\lambda_k(-,\underline{V})}\big)\to \holim_{\mathcal{P}_0(\{k\}_+)}\dots \holim_{\mathcal{P}_0(\{1\}_+)}  NX_{\lambda(-,\underline{V})}\]
is an equivalence for every $\underline{V}=\{V_i\}_{i=k+1}^n$. We need to show that for every collection $\underline{U}=\{U_i\}_{i=k+2}^n$ the nerve of $\Hom\big(\mathcal{P}_0((k+1)_+)/_{(-)},X_{\lambda_{k+1}(-,\underline{U})}\big)$ is equivalent to
\[\holim\Big( \holim\limits_{\prod_{i=1}^k\mathcal{P}_0(\{i\}_+)}NX_{\lambda(-,+,\underline{U})}\to \holim\limits_{\prod_{i=1}^k\mathcal{P}_0(\{i\}_+)}NX_{\lambda(-,\{k+1\}_+,\underline{U})}\leftarrow \holim\limits_{\prod_{i=1}^k\mathcal{P}_0(\{i\}_+)}NX_{\lambda(-,k+1,\underline{U})} \Big)\]
For every non-empty subset $W$ of $k_+$, the zig-zag $\lambda_k(W,+,\underline{U})\to\lambda_k(W,\{k+1\}_+,\underline{U})\leftarrow \lambda_k(W,k+1,\underline{U})$ in the category $\mathcal{P}_0(n_+)$ inducing the pullback diagram above is
\[\big((\underline{n}\backslash \underline{k+1})\cup W\big)\cap r_{k+1}(\underline{U})\to \big((\underline{n}\backslash \underline{k})\cup W\big)\cap r_{k+1}(\underline{U})\leftarrow \big((\underline{n}\backslash \underline{k})\cup (W\cap \underline{n})\big)\cap r_{k+1}(\underline{U})\]
The right-hand term is invariant under adding the basepoint $+$ to $W$. Hence all the maps $W\to W_+$ in the punctured $k_+$-cube $NX_{\lambda_k(-,k+1,\underline{U})}$ are identities. The right hand homotopy limit is then naturally equivalent to the $+$-vertex, that is
\[NX_{\lambda_{k}(+,k+1,\underline{U})}\stackrel{\simeq}{\longrightarrow} \holim\limits_{\prod_{i=1}^k\mathcal{P}_0(\{i\}_+)}NX_{\lambda(-,k+1,\underline{U})}\]
Together with the induction hypothesis, this shows that the homotopy pullback above is equivalent to
\[\holim\left(\vcenter{\hbox{\xymatrix{ N\Hom\big(\mathcal{P}_0(k_+)/_{(-)},X_{\lambda_k(-,+,\underline{U})}\big)\ar[d]^{f_+}\\ N\Hom\big(\mathcal{P}_0(k_+)/_{(-)},X_{\lambda_k(-,\{k+1\}_+,\underline{U})}\big)\\ NX_{\lambda_{k}(+,k+1,\underline{U})}\ar[u]_{f_{k+1}}}}}\right)\]
An argument completely analogous to Lemma \ref{indgrot} gives a natural isomorphism of categories 
\[\Big(\Hom\big(\mathcal{P}_0(k_+)/_{(-)},X_{\lambda_k(-,\{k+1\}_+,\underline{U})}\big)\mbox{$\int$} \big(f_+/_{(-)}\times f_{k+1}/_{(-)}\big)\Big)\cong \Hom\big(\mathcal{P}_0((k+1)_+)/_{(-)},X_{\lambda_{k+1}(-,\underline{U})}\big)\]
By Quillen's theorem $B_2$ of \cite{ClarkKan} this Grothendieck construction is equivalent to the homotopy pullback, if $(f_{k+1})/_{(-)}$ sends morphisms to equivalences.
If $+$ belongs to $r_{k+1}(\underline{U})$ the whole set $\underline{k+1}_+$ is contained in $r_{k+1}(\underline{U})$, and
\[\lambda_k(W,\{k+1\}_+,\underline{U})=\big((\underline{n}\backslash \underline{k})\cup W\big)\cap r_{k+1}(\underline{U})=\big((\underline{n}\backslash \underline{k})\cap r_{k+1}(\underline{U})\big)\cup W\]
Hence the punctured $k_+$-cube $X_{\lambda_k(-,\{k+1\}_+,\underline{U})}$ is the translation cube $X_{K\cup (-)}$, for the subset $K=(\underline{n}\backslash \underline{k})\cap r_{k+1}(\underline{U})$ of $\underline{n}\backslash \underline{k}$. Therefore $f_{k+1}/_{(-)}=c_K/_{(-)}$ sends morphisms to equivalences by assumption. If $+$ is not in $r_{k+1}(\underline{U})$ the maps $W\to W_+$ in the punctured cube $X_{\lambda_k(-,\{k+1\}_+,\underline{U})}$ are identities, and $f_{k+1}$ is the map from the $+$-vertex
\[f_{k+1}\colon X_{\lambda_{k}(+,k+1,\underline{U})}=X_{\lambda_{k}(+,\{k+1\}_+,\underline{U})}\stackrel{\simeq}{\longrightarrow}
\Hom\big(\mathcal{P}_0(k_+)/_{(-)},X_{\lambda_k(-,\{k+1\}_+,\underline{U})}\big)\]
This is an equivalence of categories by the cofinality theorem of \cite{Lydakis}, and therefore $f_{k+1}/_{(-)}$ sends morphisms to equivalences.
\end{proof}

\section{Equivariant Left-finite homotopy limits of categories}

In this section we prove Theorem $B^{I}_G$ stated in the introduction, for a finite group $G$. This is an equivariant version of Theorem $B^I$ for diagrams of categories that are equipped with a $G$-action. We make all the constructions of Theorem $B^{I}_G$ described in the introduction precise, by reviewing a few notions from \cite{Jack} and \cite{Gdiags}.

Let $G$ be a finite group acting on a category $I$. This is a functor $I\colon G\to Cat$, where the group $G$ is seen as a category with one object. By abuse of notation we will write $I$ also for the underlying category. We recall from \cite{Jack} that a $G$-diagram in an ambient category $\mathcal{C}$ is a diagram $X\colon I\to \mathcal{C}$ together with a natural transformation $g_X\colon X\to X\circ g$ for every group element $g\in G$, which satisfy $g_X\circ h_X=(gh)_X$ and $1_X=\id$. This is the same data as a diagram $X\colon G\int I\to\mathcal{C}$, with domain the Grothendieck construction of the $G$-action $I\colon G\to Cat$. Notice that each vertex $X_i$ has an action by the stabilizer group $G_i$ of the object $i$.

\begin{ex}
The diagram $I/_{(-)}\colon I\to Cat$ has a canonical structure of $G$-diagram, defined by the natural transformations $g_\ast\colon I/_{i}\to I/_{gi}$ induced by the automorphisms $g$ of $I$. Similarly the nerve of this diagram $NI/_{(-)}\colon I\to sSet$ is a $G$-diagram of simplicial sets.
\end{ex}

When $\mathcal{C}$ is the category of simplicial sets, the simplicial set of natural transformations $\Hom(K,Z)$ between two $G$-diagrams $K$ and $Z$ inherits a $G$-action by conjugation. An $n$-simplex of $\Hom(K,Z)$ is a natural transformation $f\colon K\times \Delta^n\rightarrow Z$. This is sent by the $G$-action to the composite
\[\xymatrix{gf\colon K\times \Delta^n\ar[rr]^-{g^{-1}_K\times\id}&&(K\circ g^{-1})\times \Delta^n=(K\times \Delta^n)\circ g^{-1}\ar[r]^-f&Z\circ  g^{-1}\ar[r]^-{g_Z}&Z}\]
In particular for the $G$-diagram $K=NI/_{(-)}$ this gives the Bousfield-Kan formula $\Hom(NI/_{(-)},Z)$ the structure of a simplicial $G$-set.

The category of $G$-diagrams in simplicial sets has a model structure where equivalences (respectively fibrations) are the natural transformations $f\colon K\to Z$ with the property that for every vertex $i\in I$ the map $f_i\colon K_i\to Z_i$ is an equivalence (respectively a fibration) of simplicial $G_i$-sets (see \cite[2.6]{Gdiags}).
 Define the homotopy limit of a $G$-diagram of simplicial sets $Z\colon I\to sSet$ to be the simplicial $G$-set
\[\holim Z:=\Hom(NI/_{(-)},FZ)\]
where $Z\stackrel{\simeq}{\to} FZ$ is a fibrant replacement of the $G$-diagram $Z$ in this model structure. This construction is interpreted in a suitable model categorical context in \cite{Gdiags}. In particular it sends weak equivalences of $G$-diagrams to weak $G$-equivalences of simplicial $G$-sets (see \cite[2.22]{Gdiags}).

Given a $G$-diagram $X\colon I\to Cat$, the category of natural transformations $\Hom(I/_{(-)},X)$ has a similar $G$-action by conjugation, and its nerve is isomorphic to $\Hom(NI/_{(-)},NX)$ as a simplicial $G$-set. Hence the fibrant replacement induces a $G$-map
\[N\Hom(I/_{(-)},X)\longrightarrow\holim_I NX\]
which is the map of the statement of Theorem $B_{G}^I$.

The proof of Theorem $B_{G}^I$ is based on the same inductive argument in the proof of theorem $B^I$. The key ingredient for the induction step is an equivariant analogue of lemma \ref{indgrot}. If $Y\colon I\to Cat$ is a $G$-diagram of categories, the Grothendieck construction $I\int Y$ has an induced $G$-action, defined on objects by
\[g\cdot \big(i\in I,c\in Y_i\big)=\big(gi, g_Y(c)\in Y_{gi}\big)\]
and sending a morphism $\big(\alpha\colon i\to j,\delta\colon \alpha_\ast c\to d\big)$ to
\[g\cdot \big(\alpha,\delta\big)=\big(g\alpha,(g\alpha)_\ast g_Y(c)= g_Y(\alpha_\ast c)\stackrel{g_Y(\delta)}{\longrightarrow}g_Y(d)\big)\]
Given a subset $U\subset I_n$, the $G$-action on $I$ induces a $G_U$-action on the categories $U\leq I$ and $U<I$, where $G_U$ is the subgroup of $g$ of elements that send $U$ to itself. The functor $F_U\colon \Hom\big((U\unders I)/_{(-)},X_{U<}\big)\to Cat$ from Lemma \ref{indgrot} that sends $\Phi\colon (U\unders I)/_{(-)}\to X_{U<}$ to 
\[F_U(\Phi)=\prod_{u\in U}(m_u)/_{(\Phi|_{u<I})}\]
has a canonical $G_U$-structure. It is defined by conjugating the $G_U$-action on $U$ indexing the product and the functors
\[\xymatrix{(m_u)/_{(\Phi|_{u<I})}\ar[r]\ar@{-->}[d]&X_u\ar[r]\ar[d]^{g}&\Hom\big((u\unders I)/_{(-)},X_{u<}\big)\ar[d]^g\\
(m_u)/_{(g\Phi|_{u<I})}\ar[r]&X_u\ar[r]&\Hom\big((u\unders I)/_{(-)},X_{u<}\big)
}\]
Hence the Grothendieck construction of $F_U$ inherits a $G_U$-action. The following is immediate.
\begin{lemma}
For every subset $U\subset I_n$, the isomorphism of categories
\[\Hom\big(U\!\leq\! I)/_{(-)},X_{U\leq}\big)\cong \Big(\Hom\big((U\unders I)/_{(-)},X_{U<}\big)\mbox{$\int$} F_U\Big)\]
of lemma \ref{indgrot} is $G_U$-equivariant.
\end{lemma}

\begin{proof}[Proof of Theorem $B_{G}^I$]

Suppose that $I$ is left finite. The automorphism $g$ of $I$ induces an isomorphism of categories $g\colon i/I\to gi/I$. It follows that the nerves $N(i/I)$ and $N(gi/I)$ have the same dimension, and that the degree function $\deg\colon Ob I\to\mathbb{N}$ is $G$-invariant. Hence the $G$-action restricts to the filtration
\[I_{\leq 0}\subset I_{\leq 1}\subset\dots\subset I_{\leq n}\subset\dots\subset I\]
and the $G$-structure on $X\colon I\to Cat$ restricts to a $G$-structure on $X_{\leq n}\colon I_{\leq n}\to Cat$.

Let $NX\stackrel{\simeq}{\to} FNX$ be a pointwise fibrant replacement of $NX$ in the model structure of $G$-diagrams of simplicial set described above. We prove by induction on $n$ that for every subset $U\subset I_n$ the map
\[N\Hom\big((U\!\leq\! I)/_{(-)},X_{U\leq}\big)\cong \Hom\big(N(U\!\leq\! I)/_{(-)},NX_{U\leq}\big)\to\holim_{U\leq I}(FNX)_{U\leq}\]
is a weak $G_U$-equivalence. Once this is established, the same argument in the proof of Theorem $B^I$ finishes the proof of $B_{G}^I$.

For $n=0$, the category $U\leq I$ is discrete and the map above is the map of indexed products
\[\prod\limits_{u\in U}NX_u\longrightarrow \prod\limits_{u\in U}FNX_u\]
The fixed points of this map by a subgroup $H\leq G_U$ is isomorphic to the map
\[\prod\limits_{[u]\in U/_H}NX_{u}^{H_{u}}\longrightarrow \prod\limits_{[u]\in U/_H}FNX_{u}^{H_{u}}\]
where for a choice of representative in each $H$-orbit of $U$, where $H_u$ is the stabilizer group of $u$ in $H$. But each map $NX_{u}^{H_{u}}\to FNX_{u}^{H_{u}}$ is an equivalence of simplicial sets by assumption, and the map above is an equivalence.

Now suppose that the claim is true for $n$, and let $V$ be a subset of $I_{n+1}$. The sequence
\[NF_{V}(\Phi)\longrightarrow N\big(\Hom\mbox{$\int$} F_V\big)\cong N\Hom\big((V\!\leq\! I)/_{(-)},X_{V\leq}\big)\longrightarrow
N\Hom\big((V\!<\! I)/_{(-)},X_{V<}\big)\]
induced by the restriction map is a fiber sequence of simplicial $G_V$-sets. This is because its restriction on fixed points of a subgroup $H\leq G_V$ is the sequence
\[NF_{V}(\Phi)^H\longrightarrow N\big(\Hom\mbox{$\int$} F_V\big)^ H\cong N \big(\Hom^H\mbox{$\int$} F^{H}_V\big)\longrightarrow
N\Hom\big((V\!<\! I)/_{(-)},X_{V<}\big)^H\]
where the functor $F^{H}_V\colon N\Hom\big((V\!<\! I)/_{(-)},X_{V<}\big)^H\to Cat$ sends an $H$-equivariant natural transformation $\Phi$ to
\[F^{H}_V(\Phi)=\Big(\prod_{v\in V}(m_v)/_{(\Phi|_{v<I})}\Big)^H\cong \prod\limits_{[v]\in V/_H}m^{H_v}_v/_{(\Phi|_{v<I})}\]
By assumption  $m^{H_v}_v/_{(-)}$ sends every morphism to a weak equivalence, and thus so does $F^{H}_V$. It follows by Lemma \cite[p.98]{Quillen} and \cite{Thomason} that $NF^{H}_V$ is indeed the homotopy fiber of the restriction map. The restriction map
\[\holim NX_{V\leq}\longrightarrow \holim NX_{V<}
\]
is a fibration of simplicial $G$-sets by an argument analogous to the one in the proof of Theorem $B^I$. Its fiber is the product of homotopy fibers $\prod\limits_{v\in V}\hof_{\Phi|_{v<I}}\big( NX_v\to \holim NX_{v<}\big)$. Therefore it remains to show that the map on homotopy fibers
\[NF_V(\Phi)\longrightarrow \prod\limits_{v\in V}\hof_{\Phi|_{v<I}}\big( NX_v\to \holim NX_{v<}\big)\]
is a $G_V$-equivalence. By taking fixed points, this is the case if for every $v\in V$ the map
\[Nm_v/_{(\Phi|_{v<I})}\longrightarrow\hof_{\Phi|_{v<I}}\big( NX_v\to \holim NX_{v<}\big)\]
is a $G_v$-equivalence. This map factors as
\[Nm_v/_{(\Phi|_{v<I})}\to\hof_{\Phi|_{v<I}}\Big( NX_v\to N\Hom\big((v\!<\!I)/_{(-)},X_{v<}\big)\Big)\to \hof_{\Phi|_{v<I}}\big( NX_v\to \holim NX_{v<}\big)\]
The first map is a $G_v$-equivalence, since $m^{H}_v/_{(-)}$ sends every morphism to a weak equivalence of categories for every subgroup $H$ of $G_v$. The second map is also a $G_v$ equivalence, as the map $N\Hom\big((v<I)/_{(-)},X_{v<}\big)\to\holim NX_{v<}$ is a $G_v$-equivalence by the inductive hypothesis.
\end{proof} 

\bibliographystyle{amsalpha}
\bibliography{GQuillen}

\providecommand{\bysame}{\leavevmode\hbox to3em{\hrulefill}\thinspace}
\providecommand{\MR}{\relax\ifhmode\unskip\space\fi MR }
\providecommand{\MRhref}[2]{%
  \href{http://www.ams.org/mathscinet-getitem?mr=#1}{#2}
}
\providecommand{\href}[2]{#2}
\begin{thebibliography}{Tho79}

\bibitem[BK72]{BK}
A.~K. Bousfield and D.~M. Kan, \emph{Homotopy limits, completions and
  localizations}, Lecture Notes in Mathematics, Vol. 304, Springer-Verlag,
  Berlin, 1972. \MR{0365573 (51 \#1825)}

\bibitem[BK13]{ClarkKan}
C.~Barwick and D.~Kan, \emph{Quillen theorems $b$n for homotopy pullbacks of
  (infinity, k)-categories}, arXiv:1208.1777, 2013.

\bibitem[DM14]{Gdiags}
E.~Dotto and K.~Moi, \emph{Homotopy theory of $g$-diagrams and equivariant
  excision}, arXiv 1403.6101, 2014.

\bibitem[Hir03]{Hirsch}
Philip~S. Hirschhorn, \emph{Model categories and their localizations},
  Mathematical Surveys and Monographs, vol.~99, American Mathematical Society,
  Providence, RI, 2003. \MR{1944041 (2003j:18018)}

\bibitem[JS01]{Jack}
Stefan Jackowski and Jolanta S{\l}omi{\'n}ska, \emph{{$G$}-functors,
  {$G$}-posets and homotopy decompositions of {$G$}-spaces}, Fund. Math.
  \textbf{169} (2001), no.~3, 249--287. \MR{1852128 (2002h:55017)}

\bibitem[Lyd94]{Lydakis}
Manos~G. Lydakis, \emph{Homotopy limits of categories}, J. Pure Appl. Algebra
  \textbf{97} (1994), no.~1, 73--80. \MR{1310749 (96a:18011)}

\bibitem[Qui10]{Quillen}
Daniel Quillen, \emph{Higher algebraic {$K$}-theory: {I} [mr0338129]},
  Cohomology of groups and algebraic {$K$}-theory, Adv. Lect. Math. (ALM),
  vol.~12, Int. Press, Somerville, MA, 2010, pp.~413--478. \MR{2655184}

\bibitem[Tho79]{Thomason}
R.~W. Thomason, \emph{Homotopy colimits in the category of small categories},
  Math. Proc. Cambridge Philos. Soc. \textbf{85} (1979), no.~1, 91--109.
  \MR{510404 (80b:18015)}

\end{thebibliography}

\end{document}